\documentclass[a4paper,11pt,oneside]{article}
\usepackage[latin1]{inputenc}
\usepackage[english]{babel}
\usepackage{amsthm}
\usepackage{amssymb}
\usepackage{amsmath}
\usepackage{hyperref}
\usepackage{graphicx}
\usepackage[all]{xy}

\DeclareMathAlphabet{\mathpzc}{OT1}{pzc}{m}{it}
\newtheoremstyle{note}{11pt}{11pt}{}{}{\bfseries}{.}{.5em}{}
\newtheorem{theo}[equation]{Theorem}
\newtheorem{prop}[equation]{Proposition}

\newtheorem{conj}[equation]{Conjecture}

\newtheorem{coro}[equation]{Corollary}
\numberwithin{equation}{section}
\newtheorem{lemma}[equation]{Lemma}

\newcommand{\Q}{\mathbb{Q}}
\newcommand{\F}{\mathbb{F}}

\newcommand{\Z}{\mathbb{Z}}
\newcommand{\Gm}{\mathbb{G}_m}
\newcommand{\R}{\mathbb{R}}
\newcommand{\C}{\mathbb{C}}
\newcommand{\Hbb}{\mathbb{H}}

\newcommand{\K}{\mathcal{K}}

\newcommand{\I}{\mathcal{I}}

\newcommand{\Ll}{\mathcal{L}}

\newcommand{\Cc}{\mathcal{C}}

\newcommand{\oo}{\mathcal{O}}
\newcommand{\h}{\mathcal{H}}

\newcommand{\W}{\mathcal{W}}
\newcommand{\U}{\mathcal{U}}

\newcommand{\G}{\Gamma}
\newcommand{\LL}{\Lambda}
\newcommand{\eps}{\varepsilon}

\newcommand{\gl}{\mathrm{GL}_2}
\newcommand{\Sl}{\mathrm{SL}_2}

\newcommand{\Log}{\mathrm{Log}}

\newcommand{\boldf}{\mathbf{f}}

\newcommand{\boldF}{\mathbf{F}}

\newcommand{\Dfrak}{\mathfrak{D}}

\newcommand{\pfrak}{\mathfrak{p}}

\newcommand{\bfrak}{\mathfrak{b}}

\newcommand{\lgr}{\left\{}
\newcommand{\rgr}{\right\}}
\newcommand{\rra}{\right\rangle}
\newcommand{\lla}{\left\langle}

\title{Derivative of symmetric square $p$-adic $L$-functions via pull-back formula}
\author{Giovanni Rosso
\footnote{PhD Fellowship of Fund for Scientific Research - Flanders, partially supported by a JUMO grant from KU Leuven (Jumo/12/032), a ANR grant  (ANR-10-BLANC 0114 ArShiFo) and a NSF grant (FRG  DMS  0854964). 
}
}

\begin{document}

\maketitle
In this paper we recall the method of Greenberg and Stevens to calculate derivatives of $p$-adic $L$-functions using deformations of Galois representation and we apply it to the symmetric square of a modular form  Steinberg at $p$. Under certain hypotheses on the conductor and the Nebentypus, this proves a conjecture of Greenberg and Benois on trivial zeros.
\tableofcontents
\section{Introduction}
Let $M$ be a motive over $\Q$ and suppose that it is pure of weight zero and irreducible. We suppose also that $s=0$ is a critical integer \`a la Deligne.\\ 
 Fix a prime number $p$ and let $V$ be the $p$-adic representation associated to $M$. We fix once and for all an isomorphism $\C \cong \C_p$. If $V$ is semistable, it is conjectured that for each regular submodule $D$ \cite[\S 0.2]{BenLinv} there exists a $p$-adic $L$-function $L_p(V,D,s)$. It is supposed to interpolate the special values of the $L$-function of $M$ twisted by finite-order characters of $1+p\Z_p$ \cite{PR}, multiplied by a corrective factor (to be thought of as a part of the local epsilon factor at $p$) which depends on $D$. In particular, we expect the following interpolation formula at $s=0$;
 \begin{align*}
 L_p(V,D,s) = E(V,D) \frac{L(V,0)}{\Omega(V)},
 \end{align*}
for $\Omega(V)$   a complex period and $E(V,D)$ some Euler type factors which conjecturally have to be removed in order to permit $p$-adic interpolation (see \cite[\S 2.3.2]{BenLinv} for the case when $V$ is crystalline). It may happen that certain of these Euler factors vanish. In this case the connection with what we are interested in, the special values of the $L$-function, is lost. Motivated by the seminal work of Mazur-Tate-Teitelbaum \cite{MTT}, Greenberg, in the ordinary case \cite{TTT}, and Benois \cite{BenLinv} have conjectured the following;
\begin{conj}\label{MainCo}[Trivial zeros conjecture]
Let $e$ be the number of Euler-type factors of $E_p(V,D)$ which vanish. Then the order of zeros at $s=0$ of $L_p(V,D,s)$ is $e$ and
\begin{align}\label{FormMC}
\lim_{s \rightarrow 0} \frac{L_p(V,D,s)}{s^e } = \Ll(V,D) E^*(V,D) \frac{L(V,0)}{\Omega(V)}
\end{align}
for $E^*(V,D) $ the non-vanishing factors of $E(V,D)$ and   $\Ll(V,D)$ a non-zero number called the $\Ll$-invariant.
\end{conj}
There are many different ways in which the $\Ll$-invariant can be defined. A first attempt at such a definition could that of an {\it analytic} $\Ll$-invariant
\begin{align*}
\Ll^{\mathrm{an}}(V,D)=\frac{\lim_{s \rightarrow 0} \frac{L_p(V,D,s)}{s^e}} {E^*(V,D) \frac{L(V,0)}{\Omega(V)}}.
\end{align*}
Clearly, with this definition, the above conjecture reduces to the statement on the order of $L_p(V,D,s)$ at $s=0$ and the non-vanishing of  the $\Ll$-invariant.\\

In \cite{MTT} the authors give a more arithmetic definition of the $\Ll$-invariant for an elliptic curve, in terms of an {\it extended} regulator on the {\it extended} Mordell-Weil group. The search for an intrinsic, Galois theoretic interpretation of this error factor led to the definition of the {\it arithmetic} $\Ll$-invariant $\Ll^{\mathrm{ar}}(V,D)$  given by Greenberg \cite{TTT} (resp. Benois \cite{BenLinv}) in the ordinary case (resp. semistable case) using Galois cohomology (resp. cohomology of $(\varphi,\G)$-module). \\
For $2$-dimensional Galois representations many more definitions have been proposed and we refer to \cite{ColSurL} for a detailed exposition.\\ 

When the $p$-adic $L$-function can be constructed using an Iwasawa cohomology class and the big exponential \cite{PR}, one can use the machinery developed in \cite[\S 2.2]{BenNC} to prove formula (\ref{FormMC}) with $\Ll=\Ll^{\mathrm{al}}$. Unluckily, it is a very hard problem to construct classes in cohomology  which are  related to  special values. Kato's Euler system has been used in this way in \cite{BenNC} to prove many instances of Conjecture \ref{MainCo} for modular forms. It might be possible that the construction of Lei-Loeffler-Zerbes \cite{LLZ} of an Euler system for the Rankin product of two modular forms could produce such Iwasawa classes for other Galois representations; in particular, for $V=\mathrm{Sym}^2(V_f)(1)$, where $V_f$ is the Galois representation associated to a weight two modular form (see also \cite{PR2} and the upcoming work of Dasgupta on Greenberg's conjecture for the symmetric square $p$-adic $L$-functions of ordinary forms). We also refer the reader to \cite{BDR2}.\\

We  present in this paper a different method which has already been used extensively in many cases and which we think to be more easy to apply at the current state: the method of Greenberg and Stevens \cite{TTT}.
Under certain hypotheses which we shall state in the next section, it allows us to calculate the derivative of ${L_p(V,D,s)}$.\\

The main ingredient of their method is the fact that $V$ can be $p$-adically deformed in a one dimensional family. For example,  modular forms can be deformed in a Hida-Coleman family. We have decided to present this method because of the recent developments on families of automorphic forms \cite{AIP,Bras} have opened the door to the construction of families of $p$-adic $L$-function in many different settings. For example, we refer to the ongoing PhD thesis of Z. Liu on $p$-adic $L$-functions for  Siegel modular forms.\\ 
Consequently, we expect  that one could prove many new instances of Conjecture \ref{MainCo}. \\ 

In Section \ref{padicLfunctions} we shall apply this method to the case of the symmetric square of a modular form which is Steinberg at $p$. The theorem which will be proved is the following;
\begin{theo}\label{MainThBon}
Let $f$ be a modular form of trivial Nebentypus, weight $k_0$ and conductor $Np$, $N$ squarefree and prime to $p$. If $p=2$, suppose then $k_0=2$. Then Conjecture \ref{MainCo} (up to the non-vanishing of the $\Ll$-invariant) is true for $L_p(s,\mathrm{Sym}^2(f))$.
\end{theo}
In this case there is only one choice for the regular submodule $D$ and the trivial zero appears at $s=k_0-1$. This theorem generalizes \cite[Theorem  1.3]{RosOC} in two ways: we allow $p=2$ in the ordinary case and when $p\neq 2$ we do not require $N$ to be even. In particular, we cover the case of the symmetric square $11$-adic $L$-function for the elliptic curve $X_0(11)$ and of the $2$-adic $L$-function for the symmetric square of $X_0(14)$. \\

In the ordinary case ({\it i.e.} for $k=2$), the proof is completly independent of \cite{RosOC} as we can construct directly a two-variable $p$-adic $L$-function. In the finite slope case we can not construct a two-variable function with the method  described below and consequently we need the two-variable $p$-adic $L$-function of \cite[Theorem 4.14]{RosOC}, which has been constructed only for $p\neq 2$.\\

In the ordinary setting, the same theorem (but with the hypothesis at $2$) has been proved by Greenberg and Tilouine (unpublished). The importance of this formula for the proof of the Greenberg-Iwasawa Main Conjecture \cite{Urb} has been put in evidence in \cite{HTU}. \\

We improve the result of \cite{RosOC} a different construction of the $p$-adic $L$-function, namely that of \cite{BS}. We express the complex $L$-function using Eisenstein series for $\mathrm{GSp}_4$, a pullback formula to the Igusa divisor and a double Petersson product. We are grateful to \'E. Urban for having suggested this approach to us. In Section \ref{Eisenstein} we briefly recall the theory of Siegel modular forms and develop a theory of $p$-adic modular forms for $\gl \times \gl$ necessary for the construction of $p$-adic families of Eisenstein series. \\

\paragraph{Acknowledgement} This paper is part of the author's PhD thesis and the author is very grateful to his director J. Tilouine for the constant guidance and the attention given. We thank \'E. Urban for inviting the author at Columbia University (where this work has seen the day) and for the generosity with which he shared with us his ideas and insights. We  thank R. Casalis for interesting discussions on K\"ahler differentials and J. Welliaveetil for useful remarks.

\section{The method of Greenberg and Stevens}\label{GS}
The aim of this section is to recall the method of Greenberg and Stevens \cite{SSS} to calculate analytic $\Ll$-invariant. This method has been used successfully many other times \cite{Mok,RosH,RosOC}. It is  very robust and easily adaptable to many situations in which the expected order of the trivial zero is one. We also describe certain obstacles which occur while trying to apply this method to higher order zeros.\\

 We let $K$ be a $p$-adic local field, $\oo$ its valuation ring and $\LL$  the Iwasawa algebra $\oo[[T]]$. Let $V$ be a $p$-adic Galois representation as before.\\
We  denote by $\W$ the rigid analytic space whose $\C_p$-points are $\mathrm{Hom}_{\mathrm{cont}}(\Z_p^{*},\C_p^*)$. We have a map $\Z \rightarrow \W$ defined by $k \mapsto [k]: z \rightarrow z^k$. Let us fix once and for all, if $p\neq 2$ (resp. $p=2$)  a generator $u$ of $1+p\Z_p$ (resp. $1+4\Z_2$) and a decomposition $\Z_p^* = \mu \times 1+p\Z_p$ (resp $\Z_2^* = \mu \times 1+4\Z_2$). Here $\mu= \Gm^{\mathrm{tors}}(\Z_p)$. We have the following isomorphism of rigid spaces:
\begin{align*}
 \W \cong & {\Gm^{\mathrm{tors}}(\Z_p)}^{\wedge} \times B(1,1^{-}) \\
  \kappa & \mapsto  (\kappa_{\vert_{\mu}}, \kappa(u)),
\end{align*}
where the first set has the discrete topology and the second is the rigid open unit ball around $1$. Let $0<r < \infty$, $r \in \R$, we define 
\begin{align*}
 \W(r)= \lgr (\zeta,z) \vert \zeta \in \Gm^{\mathrm{tors}}(\Z_p), \vert z - 1 \vert \leq p^{-r}  \rgr.
\end{align*}
We fix an integer $h$ and we denote by $\h_h$ the algebra of $h$-admissible distributions over $1+p\Z_p$ (or $1+4\Z_2$ if $p=2$) with values in $K$. Here we take the definition of admissibility as in \cite[\S 3]{Pan}, so measures are one-admissible, i.e. $\LL \otimes \Q_p \cong \h_1 $. The Mellin transform gives us a map
\begin{align*}
 \h_h \rightarrow \mathrm{An(0)}, 
\end{align*}
 where $\mathrm{An(0)}$ stands for the algebra of $\Q_p$-analytic, locally convergent functions around $0$. If we see $\h_h$ as a subalgebra of the ring of formal series, this amounts to $T \mapsto u^s-1$.\\ 
 
We suppose that we can construct a $p$-adic $L$-function for $L_p(s,V,D)$ and that it presents a single trivial zero.\\
We suppose also that $V$ can be deformed in a $p$-adic family $V(\kappa)$. Precisely, we suppose that we are given an affinoid $\U$, finite over  $\W(r)$. Let us write $\pi:\U \rightarrow \W(r)$. Let us denote by $\I$ the Tate algebra corresponding to $\U$. We suppose that $\I$ is integrally closed and that there exists a {\it big} Galois representation $V(\kappa)$ with values in $\I$ and a point $\kappa_0 \in \U$ such that $V=V(\kappa_0)$. \\
We define  $\U^{\mathrm{cl}}$ to be the set of $\kappa \in \U$ satisfying the followings conditions:
\begin{itemize}
 \item $\pi(\kappa)=[k]$, with $k \in \Z$,
\item $V(\kappa)$ is motivic, 
\item $V(\kappa)$ is semistable as $G_{\Q_p}$-representation,
\item $s=0$ is a critical integer for $V(\kappa)$.
\end{itemize}
 We make the following assumption on $\U^{\mathrm{cl}}$.\\
 
(\textbf{CI}) For every $n > 0$, there are infinitely many $\kappa$ in $\U^{\mathrm{cl}}$ and $s \in \Z$ such that:
\begin{itemize}
\item  $\vert \kappa -\kappa_0 \vert_{\U} < p^{-n} $ 
 \item  $s$ critical for $V(\kappa)$, 
 \item $s\equiv 0 \bmod p^n$. 
\end{itemize}

This amounts to asking that the couples $(\kappa,[s])$ in $\U \times \W$ with $s$ critical for $V(\kappa)$ accumulate at $(\kappa_0,0)$.\\

We suppose that there is a global triangulation $D(\kappa)$ of the $(\varphi,\G)$-module associated to $V(\kappa)$ \cite{Liu} and that this induces the regular submodule used to construct $L_p(s,V,D)$.\\
Under these hypotheses, it is natural to conjecture the existence of a two-variable $p$-adic $L$-function (depending on $D(\kappa)$) $L_p(\kappa,s)  \in \I\hat{\otimes}\h_h $ interpolating the $p$-adic $L$-functions of $V(\kappa)$, for all $\kappa$ in $\U^{\mathrm{cl}}$. Conjecturally \cite{PanTot,Pott}, $h$ should be defined solely in terms of the $p$-adic Hodge theory of  $V(\kappa)$ and $D(\kappa)$.\\ 

We make two hypotheses on this $p$-adic $L$-function.
\begin{itemize}
\item [{\it i})] There exists a subspace of dimension one $(\kappa,s(\kappa))$  containing $(\kappa_0,0)$ over which $L_p(\kappa,s(\kappa)) $ vanishes identically.
\item [{\it ii})] There exists an {\it improved} $p$-adic $L$-function $L_p^*(\kappa)$ in $\I$ such that $L_p(\kappa,0)=E(\kappa)L^*_p(\kappa)$, for $E(\kappa)$ a non-zero element which vanishes at $\kappa_0$.
\end{itemize}
The idea is that {\it i)} allows us to express the derivative we are interested in in terms of the ``derivative with respect to $\kappa$.'' The latter can be calculated using {\it ii)}. In general, we expect that $s(\kappa)$ is a simple function of $\pi(\kappa)$.\\

Let $\log_p(z)=-\sum_{n=1}^{\infty} \frac{{(1-z)}^n}{n}$, for ${\vert z - 1 \vert}_p < 1$ and 
\begin{align*}
\Log_p(\kappa) = \frac{\log_p(\kappa(u^r))}{\log_p(u^r)},
\end{align*}
for $r$ any integer big enough.\\
For example, in \cite{TTT,Mok} we have $s(\kappa)=\frac{1}{2}\Log_p(\pi(\kappa))$ and in \cite{RosH,RosOC} we have $s(\kappa)=\Log_p(\pi(\kappa))-1$. In the first case the line corresponds to the vanishing on the central critical line which is a consequence of the fact that the $\eps$-factor is constant in the family. In the second case, the vanishing is due to a line of trivial zeros, as all the motivic specializations present a trivial zero.\\

 The idea behind {\it ii)} is that the Euler factor which brings the trivial zero for $V$ varies analytically along $V(\kappa)$ once one fixes the cyclotomic variable. This is often the case with one dimensional deformations. If we allow deformations of $V$ in more than one variable, it is unlikely that the removed Euler factors define  $p$-analytic functions, due to the fact that {\it eigenvalues of the crystalline Frobenius do not vary $p$-adically} or equivalently,  that the Hodge-Tate weights are not constant.\\
 
 We now give  the example of families of Hilbert modular forms. For simplicity of notation, we consider a totally real field $F$ of degree $d$ where $p$ is split. Let $\boldf$ be a Hilbert modular form of weight $(k_1,\ldots, k_d)$ such that the parity of $k_i$ does not depend on $i$. We define $m=\mathrm{max}(k_i-1)$ and $v_i = \frac{m+1-k_i}{2}$.  We suppose that $\boldf$ is nearly-ordinary \cite{Hr2} and let $\boldF$ be the only Hida family to which $\boldf$ belongs.  For each $p$-adic place $\pfrak_i$ of $F$, the corresponding  Hodge-Tate weights are $(v_i, m-v_i)$. This implies that the Fourier coefficient $a_{\pfrak_i}(\boldF)$ is a $p$-adic analytic function only if it is divided by $p^{v_i}$. Unluckily, $a_{\pfrak_i}(\boldF)$ is the number which appears in the Euler type factor of the evaluation formula for the $p$-adic $L$-function of $\boldF$ or $\mathrm{Sym}^2(\boldF)$. This is why in \cite{Mok,RosH} the authors deal only with forms of parallel weight. It seems very hard to generalize the method 
of 
Greenberg and Stevens to higher order derivatives without new ideas.\\

It may happen that the Euler factor which brings the trivial zero for $V$ is (locally) identically zero on the whole family; this is the case for the symmetric square of a modular form of prime-to-$p$ conductor and more generally for the standard $L$-function of parallel weight Siegel modular forms of prime-to-$p$ level. That's why in \cite{RosH,RosOC} and in this article we can deal only with forms which are Steinberg at $p$.\\

We have seen in the examples above that $s(\kappa)$ is a linear function of the weight. Consequently, one needs to evaluate the $p$-adic $L$-function $L_p(V(\kappa),s)$ at $s$ which are big for the archimedean norm. When $s$ is not a critical integer it is quite a hard problem to evaluate the $p$-adic $L$-function. This is why we have supposed ({\bf CI}). It is not a hypothesis to be taken for granted. One example is the spinor $L$-function for genus two Siegel modular forms of any weight, which has only one critical integer.\\
The improved $p$-adic $L$-function is said so because $L^*_p(\kappa_0)$ is supposed to be exactly the special value we are interested in.\\

The rest of the section is devoted to make precise the expression ``derive with respect to $\kappa$.'' \\ 
We recall some facts about differentials. We fix a $\LL$-algebra $\I_1$. We suppose that $\I_1$ is a DVR and a $K$-algebra.  Let $\I_2$ be an integral domain and a local ring which is finite, flat and integrally closed over $\I_1$. We have the first fundamental sequence of K\"ahler differentials
\begin{align*}
 \Omega_{\I_1/K}\otimes_{\I_1} {\I_2} \rightarrow \Omega_{\I_2/K} \rightarrow \Omega_{\I_2/\I_1} \rightarrow 0.  
\end{align*}
Under the hypotheses above, we can write ${I_2} = \frac{I_1[X]}{P(X)}$. Then, every $K$-linear derivation of $\I_1$ can be extended to a derivation of $\I_2$ and \cite[Theorem 57 ii)]{Matsu} ensures us that the first arrow is injective.\\

Let $P_0$ be the prime ideal of $\I$ corresponding to the point $\kappa_0$ and $P$ the corresponding ideal in $\oo(\W(r))$. We take $\I_1={\oo(\W(r))}_{P}$ and $\I_2=\I_{P_0}$. The assumption that $\I_2$ is integrally closed is equivalent to ask that $\U$ is smooth at $\kappa_0$. In many cases, we expect that $\I_1 \rightarrow \I_2$ is \'etale; under this hypothesis, we can appeal to the fact that locally convergent series are Henselian \cite[Theorem 45.5]{Nag} to define  a morphism $\I_2 \rightarrow \mathrm{An}(k_0)$ to the ring of meromorphic functions around $k_0$ which extend the natural inclusion of $\I_1$. Once this morphism is defined, we can derive elements of $\I_2$ as if they were locally analytic functions. \\ 

There are some cases in which this morphism is known not to be \'etale; for example, for certain weight one forms \cite{DimDur,DimGha} and some critical CM forms \cite[Proposition 1]{BelCM}. (Note that in these cases no regular submodule $D$ exists.) \\

We would like to explain what we can do in the case when $\I_1 \rightarrow \I_2$ is not \'etale. We also hope that what we say will clarify the situation in the case where the morphism is \'etale.\\
We have that $ \Omega_{\LL/K}$ is a free rank $1$ $\LL$-module. Using the universal property of differentials 
\begin{align*}
 \mathrm{Hom}_{\LL}(\Omega_{\LL/K}, \LL)=\mathrm{Der}_{K}(\LL,\LL),
\end{align*}
 we shall say, by slight abuse of notation, that $\Omega_{\LL/K}$ is generated as $\LL$-module by the derivation $\frac{\textup{d}}{\textup{d}T}$. Similarly, we identify $\Omega_{\I_1/K}$ with the free $\I_1$-module generated by $\frac{\textup{d}}{\textup{d}T}$. As the first arrow in the first fundamental sequence is injective, there exists an element $d \in \Omega_{\I_2/K}$ (which we see as $K$-linear derivation from $\I_2$ to $\I_2$) which extends $\frac{\textup{d}}{\textup{d}T}$. If $\I_1 \rightarrow \I_2$ is \'etale, then $\Omega_{\I_2/\I_1}=0$ and the choice of $d$ is unique.\\
 
 Under the above hypotheses, we can then define a new analytic $\Ll$-invariant (which a priori depends on the deformation $V(\kappa)$ and $d$) by 
\begin{align*}
 \Ll^{\mathrm{an}}_d (V):= -{\log_p(u)}\left.{d({s(\kappa)})^{-1} d(E(\kappa))}\right\vert_{\kappa=\kappa_0} = -\left.d({s(\kappa)})^{-1}\frac{d(L_p(0,\kappa))}{L^*_p(\kappa)}\right\vert_{\kappa=\kappa_0}.
\end{align*}
We remark that with the notation above we have $\log_p(u)\frac{\textup{d}}{\textup{d}T}= \frac{\textup{d}}{\textup{d}s}$. We apply $d$ to $L_p(\kappa,s(\kappa)) = 0$ to obtain
\begin{align*}
0=  d (L_p(\kappa,s(\kappa)))  =  {\log_p(u)}^{-1} \left.\frac{\textup{d}}{\textup{d}s} L_p(\kappa,s)\right\vert_{s=s(\kappa)} d({s(\kappa)}) + \left.d(L_p(\kappa,s))\right\vert_{s=s(\kappa)}.
\end{align*}
Evaluating at $\kappa=\kappa_0$ we deduce 
\begin{align*}
\left.\frac{\textup{d}}{\textup{d}s} L_p(V,s)\right\vert_{s=0} = &  \Ll^{\mathrm{an}}_d (V) L^*_p(\kappa_0)
\end{align*}
and consequently 
\begin{align*}
 \Ll^{\mathrm{an}} (V) = \Ll^{\mathrm{an}}_d (V).
\end{align*}

In the cases in which $\Ll^{\mathrm{al}}$ has been calculated, namely symmetric powers of Hilbert modular forms, it is expressed in terms of the logarithmic derivative of Hecke eigenvalues at $p$ of certain finite slope families  \cite{BenLinv2,HarJo, HIwa, MokLinv}. Consequently, the above formula should allow us to prove  the equality $\Ll^{\mathrm{an}}=\Ll^{\mathrm{al}}$.\\
Moreover, the fact that the $\Ll$-invariant is a derivative of a non constant function shows that $\Ll^{\mathrm{an}} \neq 0$ outside a codimension $1$ subspace of the weight space. In this direction, positive results for a given $V$ have been obtained only in the cases of a Hecke character of a quadratic imaginary field and of an elliptic curve with $p$-adic uniformization, using a deep theorem in transcendent number theory \cite{Saint}.

 \section{Eisenstein measures}\label{Eisenstein}
In this section we first fix the notation concerning genus two Siegel forms. We then recall a normalization of certain Eisenstein series for $\gl \times \gl$ and develop a theory of $p$-adic families of modular forms (of parallel weight) on $\gl \times \gl$. Finally we construct two Eisenstein measures  which will be used in the next section to construct two $p$-adic $L$-functions.
\subsection{Siegel modular forms}
We now recall the basic theory of Siegel modular forms. We follow closely the notation of \cite{BS} and we refer to the first section of {\it loc. cit.} for more details. 
Let us denote by $\Hbb_1$ the complex upper half-plane and by $\Hbb_2$ the Siegel space for $\mathrm{GSp}_4$. We have explicitly
\begin{align*}
 \Hbb_2 = \lgr \left.Z = \left( \begin{array}{cc}
                           z_1 & z_2 \\
                           z_3 & z_4
                          \end{array}
\right) \right\vert  {Z}^t= Z \mbox{ and } \mathrm{Im}(Z) >0 \rgr.
\end{align*}
It has a natural action of $\mathrm{GSp}_4(\R)$ via fractional linear transformation; for any $M = \left( \begin{array}{cc}
                           A & B \\
                           C & D
                          \end{array}
\right)$ in $\mathrm{GSp}^+_4(\R)$ and $Z$ in $\Hbb_2$ we define 
\begin{align*}
 M(Z)= (AZ+B){(CZ+D)}^{-1}.
\end{align*}
Let $\G=\G^{(2)}_0(N)$ be the congruence subgroup of $\mathrm{Sp}_4(\Z)$ of matrices whose lower block $C$ is congruent to $0$ modulo $N$. We consider the space  $M^{(2)}_k(N,\phi)$ of scalar Siegel forms of weight $k$ and Nebentypus $\phi$:
\begin{align*}
 \lgr F: \Hbb_2 \rightarrow \C \left\vert  F(M(Z)){(CZ+D)}^{-k} = \phi(M)F(Z) \;\:\forall M \in \G, f \textnormal{ holomorphic} \right.\rgr .
\end{align*}
Each $F$ in $M_k(\G,\phi)$ admits a Fourier expansion 
\begin{align*}
 F(Z)= \sum_{T} a(T)e^{2 \pi i \mathrm{tr}(TZ)},
\end{align*}
where $T=\left(  \begin{array}{cc} T_1 & T_2 \\
T_2 & T_4                                                                                                                                                            \end{array}
\right)$ ranges over all matrices $T$ positive and semi-defined, with $T_1$, $T_4$ integer and $T_2$ half-integer.\\

 We have two embeddings (of algebraic groups) of $\Sl$ in $\mathrm{Sp}_4$:
\begin{align*}
 \Sl^{\uparrow}(R) = & \lgr \left( \begin{array}{cccc}
                           a & 0 & b & 0  \\
                           0 & 1 & 0 & 0 \\
c & 0 & d & 0  \\
                           0 & 0 & 0 & 1 
                          \end{array}  \right) \left\vert  \left( \begin{array}{cc}
                           a & b \\
                           c & d
                          \end{array}
\right) \right. \in \Sl(R)\rgr, \\
 \Sl^{\downarrow}(R) = & \lgr \left( \begin{array}{cccc}
                           1 & 0 & 0 & 0  \\
                           0 & a & 0 & b \\
0 & 0 & 1 & 0  \\
                           0 & c & 0 & d
                          \end{array}  \right) \left\vert  \left( \begin{array}{cc}
                           a & b \\
                           c & d
                          \end{array}
\right) \right.\in \Sl(R)\rgr .
\end{align*}
We can embed $\Hbb_1 \times \Hbb_1$ in $\Hbb_2$ in the following way
\begin{align*}
 (z_1, z_4) \mapsto  \left( \begin{array}{cc}
                           z_1 & 0 \\
                           0 & z_4
                          \end{array}
\right).
\end{align*}
If $\gamma$ belongs to $\Sl(\R)$, we have 
\begin{align*}
 \gamma^{\uparrow} \left( \begin{array}{cc}
                           z_1 & 0 \\
                           0 & z_4
                          \end{array}
\right) = \left( \begin{array}{cc}
                           \gamma(z_1) & 0 \\
                           0 & z_4
                          \end{array}
\right),\\
\gamma^{\downarrow} \left( \begin{array}{cc}
                           z_1 & 0 \\
                           0 &   z_4
                          \end{array}
\right) = \left( \begin{array}{cc}
                           z_1 & 0 \\
                           0 & \gamma (z_4)
                          \end{array}
\right).
\end{align*}
Consequently, evaluation at $z_2=0$ gives us a map
\begin{align*}
 M^{(2)}_k(N,\phi) \hookrightarrow M_k(N,\phi)\otimes_{\C} M_k(N,\phi),
\end{align*}
where $ M_k(N,\phi)$ denotes the space of elliptic modular forms of weight $k$, level $N$ and Nebentypus $\phi$.\\
This also induces a closed embedding of two copies of the modular curve in the Siegel threefold. We shall call its image the Igusa divisor. On points, it corresponds to abelian surfaces which decompose as the product of two elliptic curves. \\

We consider the following differential operators on $\Hbb_2$:
\begin{align*}
 \partial_1 = \frac{\partial}{\partial z_1},  \:\:\;\;  \partial_2 =\frac{1}{2} \frac{\partial}{\partial z_2},  \:\:\;\;  \partial_4 = \frac{\partial}{\partial z_4}.
\end{align*}
We define
\begin{align*}
 \Dfrak_l = &  z_2 \left(\partial_1\partial_4 - \partial_2^2 \right) - \left(l - \frac{1}{2} \right)\partial_2,\\
 \Dfrak_l^{s} = & \Dfrak_{l+s -1} \circ \ldots \circ \Dfrak_l, \\
 \mathring{\Dfrak}_l^{s} = &  \Dfrak_l^{s} \vert_{z_{2}=0}.
\end{align*}
The importance of $\mathring{\Dfrak}_l^{s}$ is that it preserves holomorphicity.\\
Let $I = \left(  \begin{array}{cc}
T_1 & T_2 \\
T_2 & T_4                                                                                                                                                            \end{array}
\right)$. We define  $\bfrak_l^{s}(I) $ to be the only homogeneous polynomial in the indeterminates $T_1,T_2,T_4$ of degree $s$ such that 
\begin{align*}
 \mathring{\Dfrak}_l^{s}e^{T_1 z_1 + 2 T_2z_2 + T_4 z_4} = \bfrak_l^{s}(I) e^{T_1 z_1 + T_4 z_4}. 
\end{align*}
We need to know a little bit more about the polynomial $\bfrak_l^{s}(I) $.
Let us write 
\begin{align*}
{\Dfrak_l^{s}}e^{T_1 z_1 + 2 T_2z_2 + T_4 z_4} = P_l^{s}(z_2,I) e^{T_1 z_1 + 2T_2z_2 + T_4 z_4}, 
\end{align*}
we have 
\begin{align}\label{recformula}
{\Dfrak_l^{s+1}}e^{\mathrm{tr}(ZI)} = \left[P_{l+s}^1(z_2,I)P_l^{s}(z_2,I) - \left( z_2  \frac{1}{4}\frac{\partial^2}{\partial z^2_{2}} + \left(l + s - \frac{1}{2} \right)\frac{1}{2}\frac{\partial}{\partial z_{2}} \right) P_l^{s}(z_2,I) \right] e^{tr(Z I)}.
\end{align}
We obtain easily 
\begin{align}\label{bP}
P_{l}^1(z_2,I)    = & \left(z_2 \mathrm{det}(I)-\left(l-\frac{1}{2}\right) T_2\right),\nonumber\\
\bfrak_l^{s}(I) = & P_l^{s}(0,I).
\end{align}
We have
\begin{align*}
 \bfrak_l^{s+1}(I)=& \bfrak_l^{s}(I)\bfrak_{l+s}^1(I) + (l+s-1/2)\partial_2 P_l^s(z_2,I)\\
   =& (l+s-1/2)(-T_2 \bfrak_l^{s}(I) + \partial_2 P_l^s(z_2,I)\vert_{z_2=0}).
\end{align*}

Let $J = \left\{ j_1,j_2, j_4 \right\}$. We shall write $\partial^{J}$ resp. $z^J$ for $\partial_1^{j_1}\partial_2^{j_2}\partial_4^{j_4}$ resp. $z_1^{j_1}z_2^{j_2 }z_4^{j_4}$. We  can write \cite[page 1381]{BS}
\begin{align}\label{polyDfrak}
 \mathring{\Dfrak}_l^{s} = \sum_{j_1 +j_2 + j_4=s}  c_{l}^{J} \partial^{J},
\end{align}
where $c_{l}^{J}=\frac{\mathring{\Dfrak}_l^{s}(z^J)}{\partial^{J}(z^J)}$.
We have easily:
\begin{align}
  {\partial^{J}(z^J)} = & j_1 ! j_2! j_4! 2^{-j_2}, \nonumber \\
 {\partial^{J}}( e^{\mathrm{tr}(ZI)}) = &  T_1^{j_1}T_2^{j_2}T_4^{j_4}( e^{\mathrm{tr}(ZI)}), \nonumber
\end{align}
 We pose
\begin{align*}
 c^{s}_l : = & c_{l}^{0,s,0},\\
 {(-1)}^{s} c^{s}_l = & \prod _{i=1}^{s}\left(l-1+s -\frac{i}{2} \right) =  2^{-s}\frac{(2l-2 +2s-1)!}{(2l-2+s-1)!}.
\end{align*}
Consequently, for $L \mid T_1$, $T_4$ and for any  positive integer $d$, we obtain  
\begin{align*}
4^{s}\bfrak_{t+1}^{s}(I) \equiv {(-1)}^{s}4^{s} \sum_{j_1+j_4 < d } c^{J}_{t+1} T_1^{j_1} T_2^{s-j_1-j_4} T_4^{j_4}\bmod L^{d}. 
\end{align*}
 
\subsection{Eisenstein series}\label{EisenGSp4}
The aim of this section is to recall certain Eisenstein series which can be used to construct the $p$-adic $L$-functions, as in \cite{BS}. In {\it loc. cit.} the authors consider certain Eisenstein series for $\mathrm{GSp}_{4g}$ whose pullback to the Igusa divisor is a \textbf{holomorphic} Siegel modular form.\\
 We now fix a (parallel weight) Siegel modular form $f$ for $\mathrm{GSp}_{2g}$. We write the standard $L$-function of $f$ as a double Petersson product between $f$ and these  Eisenstein series (see Proposition \ref{InteExpr}). When $g=1$, the standard $L$-function of $f$ coincides, up to a twist, with the symmetric square $L$-function of $f$ we are interested in.\\

In general, for an algebraic group bigger than $\gl$ it is quite hard to find the normalization of the Eisenstein series which maximizes, in a suitable sense, the $p$-adic behavior of its Fourier coefficients.
In \cite[\S 2]{BS} the authors develop a {\it twisting method} which allow them to define Eisenstein series whose Fourier coefficients satisfy Kummer's congruences when the character associated with the Eisenstein series varies $p$-adically. 
This is the key for their construction of the one variable (cyclotomic) $p$-adic $L$-function and of our two-variable $p$-adic $L$-function.\\
 When the character is trivial modulo $p$ there exists a simple relation between the twisted and the not-twisted Eisenstein series \cite[\S 6 Appendix]{BS}. To  construct the improved $p$-adic $L$-function, we shall simply interpolate the not-twisted Eisenstein series.\\
 
Let us now recall  these Fourier developments. \\
We fix a weight $k$, an integer $N$ prime to $p$ and a Nebentypus $\phi$. Let $f$ be an eigenform in $M_k(Np,\phi)$, of finite slope for the Hecke operator $U_p$. We write $N=N_{\mathrm{ss}} N_{\mathrm{nss}}$, where $N_{\mathrm{ss}}$ (resp. $N_{\mathrm{nss}}$) is divisible by all  primes $q \mid N$ such that $U_q f = 0$ (resp. $U_q f \neq 0$). 
Let $R$ be an integer coprime with $N$ and $p$ and $N_1$ a positive integer such that $N_{\mathrm{ss}} \mid N_1 \mid N$.
 We fix a Dirichlet character $\chi$ modulo $N_1Rp$ which we write  as $\chi_1\chi' \eps_1$, with $\chi_1$ defined modulo $N_1$, $\chi'$  primitive modulo $R$ and $\eps_1$ defined modulo $p$. We shall explain after Proposition \ref{InteExpr} why we introduce $\chi_1$.\\
Let $t\geq 1$ be an integer and ${\F^{t+1}\left(w,z,R^2 N^2 p^{2n},\phi,u\right)}^{(\chi)}$ be the twisted Eisenstein series of \cite[(5.3)]{BS}. We define
\begin{align*}
\h'_{L,\chi}(z,w) := & L(t+1+2s,\phi\chi) \mathring{\Dfrak}_{t+1}^{s}\left({\F^{t+1}\left(w,z,R^2 N^2 p^{2n},\phi,u\right)}^{(\chi)}\right) \left\vert^z U_{L^2} \right. \left\vert^w U_{L^2}\right. 
\end{align*}
for $s$ a non-negative integer and $p^{n} \mid L$, with $L$ a $p$-power. It is a form for $\G_0(N^2R^2p) \times \G_0(N^2R^2p)$ of weight $t+1+s$.\\
We shall sometimes choose $L=1$ and in this case the level is $N^2R^2p^{2n}$.\\
For any prime number $q$ and matrices $I$ as in the previous section, let $B_q(X,I)$ be the polynomial  of degree at most $1$ of \cite[Proposition 5.1]{BS}. 
We pose
\begin{align*}
 B(t) & = (-1)^{t+1} \frac{2^{1+2t}}{\G(3/2)} \pi^{\frac{5}{2}}.
\end{align*}
We deduce easily from \cite[Theorem 7.1]{BS} the following theorem 
\begin{theo}
 The  Eisenstein series defined above has the following Fourier development;
\begin{align*}
 \h'_{L,\chi}(z,w)\vert_{u=\frac{1}{2}-t} =& B(t) {(2 \pi i )}^{s}G\left({\chi}\right) \sum_{T_1 \geq 0} \sum_{T_4 \geq 0 } \\
 & \left( \sum_I \bfrak_{t+1}^{s}(I) {\left(\chi\right)}^{-1}(2T_2) \sum_{G \in \mathrm{GL}_2(\Z) \setminus \mathbf{D}(I)}(\phi\chi)^2(\mathrm{det}(G)){\vert \mathrm{det}(G) \vert}^{2t-1} \right. \\
& \left. \;\:\;\: L(1-t,\sigma_{-\mathrm{det}(2I)}\phi\chi)\prod_{q \mid\mathrm{det}(2G^{-t}IG^{-1}) } B_q\left(\chi\phi(q)q^{t-2},G^{-t}IG^{-1}\right) \right) e^{{2 \pi i}(T_1 z + T_2 w)} ,
\end{align*}
where the sum over $I$ runs along the matrices $\left(  \begin{array}{cc}
L^2 T_1 & T_2 \\
T_2 & L^2 T_4                                                                                                                                                            \end{array}
\right)$  positive definite and with $2 T_2 \in \Z$,  and 
\begin{align*}
 \mathbf{D}(I) = \left\{ G \in M_2(\Z) \vert G^{-t}IG^{-1} \textnormal{ is a half-integral symmetric matrix} \right\}.
\end{align*}
\end{theo}
\begin{proof}
The only difference from {\it loc. cit.} is that we do not apply $\left\vert\left(  \begin{array}{cc}
1 & 0 \\
0 & N^2S                                                                                                                                                            \end{array}
\right) \right.$.
\end{proof}
In fact, contrary to \cite{BS}, we prefer to work with $\G_0(N^2S)$ and not with the opposite congruence subgroup.\\
In particular each sum over $I$ is finite because $I$ must have positive determinant. Moreover, we can rewrite it as a sum over $T_2$, with $(2T_2,p)=1$ and $T_1T_4 - T^2_2 >0$.\\

It is proved in \cite[Theorem 8.5]{BS} that (small modifications of) these functions $\h'_{L,\chi}(z,w)$ satisfy Kummer's congruences. The key fact is what  they call the twisting method \cite[(2.18)]{BS}; the Eisenstein series ${\F^{t+1}\left(w,z,R^2N^2p^{2n},\phi,u \right)}^{(\chi)} $ are obtained  weighting ${\F^{t+1}\left(w,z,R^2N^2p^{2n},\phi,u \right)}$ with respect to $\chi$ over integral matrices modulo $NRp^n$. To ensure these Kummer's congruences, even when $p$ does not divide the conductor of $\chi$, the authors are forced to consider $\chi$ of level divisible by $p$. Using nothing more than Tamagawa's rationality theorem for $\gl$, they find the relation \cite[(7.13')]{BS}:
\begin{align*}
{\F^{t+1}\left(w,z,R^2N^2p, \phi,u\right)}^{(\chi)} = {\F^{t+1}\left(w,z,R^2N^2p, \phi,u\right)}^{(\chi_1\chi')}\left\vert \left(\mathrm{id} - p\left(  \begin{array}{cc}
1 & p \\
0 & 1                                                                                                                                                      \end{array}
\right)  \right) \right. .
\end{align*}
So the Eisenstein series we want to interpolate to construct the improved $p$-adic $L$-function is 
\begin{align*}
{\h'}^*_{L,\chi'}(z,w) := & L(t+1+2s,\phi\chi) \mathring{\Dfrak}_{t+1}^{s}\left({\F^{t+1}\left(w,z,R^2 N^2 p,\phi,u\right)}^{(\chi_1\chi')}\right). 
\end{align*}
In what follows, we shall specialize  $t=k - k_0+1$ (for $k_0$ the weight of the form in the theorem of the introduction) to construct the {\it improved} one variable $p$-adic $L$-function.\\
For each prime $q$, let us denote by $\alpha_q$ and $\beta_q$ the roots of the Hecke polynomial at $q$ associated to $f$. We define
\begin{align*}
D_q(X):=(1-\alpha_q^2X)(1-\alpha_q \beta_q X)(1-\beta_q^2 X).
\end{align*}
For each Dirichlet character $\chi$ we define 
\begin{align*}
\Ll(s,\mathrm{Sym}^2(f),\chi):=\prod_{q}{D_q(\chi(q)q^{-s})}^{-1}.
\end{align*}
This $L$-function differs from the motivic $L$-function $L(s,\mathrm{Sym}^2(\rho_f)\otimes \chi)$ by a finite number of Euler factors at prime dividing $N$.
We conclude with the integral formulation of $\Ll(s,\mathrm{Sym}^2(f),\chi)$ \cite[Theorem 3.1, Proposition 7.1 (7.13)]{BS}. 
\begin{prop}\label{InteExpr}
 Let $f$ be a form of weight $k$, Nebentypus $\phi$. We  put $t+s = k-1$, $s_1=\frac{1}{2}-t$ and $\h'= \h'_{1,\chi}(z,w)\vert_{u=s_1}$; we have
\begin{align*}
\lla f(w)\left\vert \left(  \begin{array}{cc}
0 & -1 \\
N^2p^{2n} R^2 & 0                                                                                                                                                      \end{array}
\right)   \right.,  \h'\rra_{N^2p^{2n} R^2}  = & \frac{\Omega_{k,s}({s_1}) p_{s_1}(t+1)}{\chi(-1) d_{s_1}(t+1)} {(RN_1 p^{n}) }^{s + 3 -k} {\left( \frac{N}{N_1}\right)}^{2-k} \times \\
& \times \Ll(s+1,\mathrm{Sym}^{2}(f),\chi^{-1})    f(z)\vert U_{N^2/N_1^2} 
\end{align*}
for \begin{align*}
\frac{ p_{s_1}(t+1)}{d_{s_1}(t+1)} = &  \frac{c_{t+1}^{s}}{c_{\frac{3}{2}}^{s}} = \frac{\prod_{i=1}^{s} \left(s + t -\frac{i}{2} \right)}{\prod_{i=1}^{s}\left(s - \frac{i -1}{2} \right) },\\
\Omega_{k,s}\left( \frac{1}{2} -t\right) = & 2^{2t} {(-1)}^{\frac{k}{2}} \pi\frac{\G\left(k - t \right)\G\left( k -  t -\frac{1}{2}\right)}{\G\left(\frac{3}{2}\right)} .
     \end{align*}
\end{prop}
\begin{proof}
With the notation of \cite[Theorem 3.1]{BS} we have $M=R^2 N^2 p^{2n}$ and $N=N_1Rp^n$. We  have that $\frac{d_{s_1}(t+1)}{p_{s_1}(t+1)}\h'$ is the holomorphic projection of the Eisenstein series of  \cite[Theorem 3.1]{BS} (see \cite[(1.30),(2.1),(2.25)]{BS}). The final remark to make is  the following relation between the standard (or adjoint) $L$-function of $f$ and the symmetric square one: \begin{align*}
\Ll(1-t,\mathrm{Ad}(f)\otimes \phi,\chi)=\Ll(k-t,\mathrm{Sym}^{2}(f),\chi).
\end{align*}
\end{proof}
The authors of  \cite{BS} prefer to work with an auxiliary character modulo $N$ to remove all the Euler factors at bad primes of $f$, but  we do not want to do this. Still, we have  to make some assumptions on the level $\chi$. Suppose that there is a prime $q$ dividing $N$ such that $q \nmid N_1$ and $U_q f=0$, then the above formula would give us zero. That is why we introduce the character $\chi_1$ defined modulo a multiple of $N_{\mathrm{ss}}$.\\
At the level of $L$-function this does not change anything as for $q \mid N_{\mathrm{ss}}$ we have $D_q(X)=1$.\\
For $f$ as in Theorem \ref{MainThBon} we can take $N_1=1$.

\subsection{Families for $\gl \times \gl$}
The aim of this section is the construction of families of modular forms on two copies of the modular curves. Let us fix a tame level $N$, and let us denote by $X$ the compactified modular curve of level $\G_1(N)\cap \G_0(p)$. For $n\geq 2$, we shall denote by $X(p^n)$ the modular curve of level $\G_1(N)\cap \G_0(p^n)$.\\ 
We denote by $X(v)$ the tube  of ray $p^{-v}$ of the ordinary locus. We fix a $p$-adic field $K$. We recall from \cite{AIS,Pil} that, for $v$ and $r$ suitable, there exists an invertible sheaf $\omega^{\kappa}$ on $X(v)\times_{K} \W(r)$. This allows us to define families of overconvergent forms as 
\begin{align*}
 M_{\kappa}(N) & := \varinjlim_v H^0(X(v)\times_{K} \W(r),\omega^{\kappa} ).
\end{align*}
We denote by $\omega^{\kappa,(2)} $  the sheaf on $X(v)\times_{K} X(v)\times_{K} \W(r)$ obtained by base change over $\W(r)$. \\ We define 
\begin{align*} 
M^{(2)}_{\kappa,v}(N) & := H^0(X(v)\times_{K} X(v)\times_{K} \W(r),\omega^{\kappa,(2)} )\\
& =H^0(X(v)\times \W(r),\omega^{\kappa}) \hat{\otimes}_{\oo(\W(r))} H^0(X(v)\times \W(r),\omega^{\kappa});\\
M^{(2)}_{\kappa}(N) & := \varinjlim_v M^{(2)}_{\kappa,v}(N).
\end{align*}
We believe that this space should correspond to families of Siegel modular forms \cite{AIP}  of parallel weight restricted on the Igusa divisor, but for shortness of exposition we do not examine this now.\\

We have a correspondence $C_p$ above $X(v)$ defined as in \cite[\S 4.2]{Pil}. We define by fiber product the correspondence $C_p^2$ on $X(v) \times_{K} X(v)$ which we extend to $X(v)\times_{K} X(v)\times_{K} \W(r)$. This correspondence induces a Hecke operator $U^{\otimes^2}_{p}$ on $H^0(X(v)\times_{K} X(v)\times_{K} \W(r),\omega^{\kappa,(2)})$ which corresponds to $U_p \otimes U_p$.
These are potentially orthonormalizable $\oo(\W(r))$-modules \cite[\S 5.2]{Pil} and  $U^{\otimes^2}_{p}$ acts on these spaces as a completely continuous operator (or compact, in the terminology of  \cite[\S 1]{Buz}) and this allows us to write 
\begin{align}
{M^{(2)}_{\kappa}(N)}^{\leq \alpha} &= \bigoplus_{\alpha_1 + \alpha_2 \leq \alpha} {M_{\kappa}(N)}^{\leq \alpha_1} \hat{\otimes}_{\oo(\W(r))} {M_{\kappa}(N)}^{\leq \alpha_2}. \label{2varslope}
\end{align}
 Here and in what follows, for  $A$ a Banach ring, $M$ a Banach $A$-module  and $U$ a completely continuous operator on $M$, we  write $M^{\leq \alpha}$ for the finite dimensional submodule of generalized eigenspaces associated to the eigenvalues of $U$ of valuation smaller or equal than $\alpha$. We write $\mathrm{Pr}^{\leq \alpha}$ for the corresponding projection. \\
We remark \cite[Lemma 2.3.13]{UrbEig} that there exists $v >0$ such that 
\begin{align*}
 {M^{(2)}_{\kappa}(N)}^{\leq \alpha}  = {M^{(2)}_{\kappa,v'}(N)}^{\leq \alpha} 
\end{align*}
for all $0 < v ' < v$. We define similarly $M^{(2)}_{\kappa}(Np^n)$.\\

We now  use the above Eisenstein series to give examples of families. More precisely, we shall construct a two-variable measure (which will be used for the two variables $p$-adic $L$-function in the ordinary case) and a one variable measure (which will be used to construct the improved one variable $L$-function) without the ordinary assumption.\\
Let us fix $\chi=\chi_1\chi'\eps_1$ as before. We suppose  $\chi$ even. We recall the Kubota-Leopoldt $p$-adic $L$-function;
\begin{theo}
 Let $\eta$ be a even Dirichlet character. There exists a $p$-adic $L$-function $L_p(\kappa,\eta) $ satisfying for any integer $t\geq 1$ and finite-order character $\eps$ of $1+p\Z_p$ 
\begin{align*}
 L_p(\eps(u)[t],\eta) = (1-(\eps\omega^{-t}\eta)_0(p))L(1-t,\eps\omega^{-t}\eta),
\end{align*}
where $\eta_0$ stands for the primitive character associated to $\eta$. If $\eta$ is not trivial then $ L_p([t],\eta)$ is holomorphic. Otherwise, it has a simple pole at $[0]$.
\end{theo}

We can consequently define a $p$-adic analytic function interpolating the Fourier coefficients of the Eisenstein series defined in the previous section; for any $z$ in  $\Z_p^*$, we define $l_z=\frac{\log_p(z)}{\log_p(u)}$. We define also 
\begin{align*}
 a_{T_1,T_4,L}(\kappa,\kappa')= & \left( \sum_I \kappa(u_{2 T_2}) {\chi}^{-1}(2T_2) \times \right. \\
 & \left. \times \sum_{G \in \mathrm{GL}_2(\Z) \setminus \mathbf{D}(I)}(\phi\chi)^2(\mathrm{det}(G)){\vert \mathrm{det}(G) \vert}^{-1}{\kappa'}^2(u^{l_{\vert \mathrm{det}(G) \vert}}) \right. \\
& \left. \;\:\;\: L_p(\kappa',\sigma_{-\mathrm{det}(2I)}\phi\chi)\prod_{q \mid\mathrm{det}(2G^{-t}IG^{-1}) } B_q\left(\phi(q)\kappa'(u^{l_q})q^{-2},G^{-t}IG^{-1}\right) \right).
\end{align*}
If $p=2$, then ${\chi}^{-1}(2T_2)$ vanishes when $T_2$ is an integer, so the above above sum is only on half-integral $T_2$ and $\kappa(u_{2 T_2})$ is a well-define $2$-analytic function. Moreover, it is $2$-integral.\\
We recall that if $p^j \mid T_1,T_4$  we have \cite[(1.21, 1.34)]{BS} 
\begin{align*}
 4^{s}\bfrak_{t+1}^{s}(I) \equiv {(-1)}^{s}4^{s}  c^{s}_{t+1}  T_2^{s} \bmod p^{d},
\end{align*}
for $s=k-t-1$. Consequently, if we define 
\begin{align*}
 \h_L(\kappa,\kappa') = \sum_{T_1 \geq 0} \sum_{T_4 \geq 0 }  a_{T_1,T_4,L}(\kappa[-1]{\kappa'}^{-1},\kappa') q_1^{T_1}q^{T_4}_2
\end{align*}
we have, 
\begin{align*}
{(-1)}^{s} 2^{s} c^{s}_{t+1} A  \h_L([k],\eps[t]) \equiv 2^{s} \h'_{L,\chi\eps\omega^{-s}}(z,w) \bmod L^r,
\end{align*}
with 
\begin{align*}
 A= A(t,k,\eps)=  B(t) {(2 \pi i )}^{s} G\left({\chi\eps\omega^{-s}}\right).
\end{align*}

We have exactly as in \cite[Definition 1.7]{Pan} the following lemma;
\begin{lemma}\label{lemmaprojinfty}
 There exists a projector \begin{align*}
 \mathrm{Pr}_{\infty}^{\leq \alpha} : \bigcup_n M^{(2)}_{\kappa}(Np^n) \rightarrow {M^{(2)}_{\kappa}(N)}^{\leq \alpha}
                          \end{align*}
which on $ M^{(2)}_{\kappa}(Np^n)$ is ${(U^{\otimes^2}_{p})}^{-i}\mathrm{Pr}^{\leq \alpha}{(U^{\otimes^2}_{p})}^{i}$, independent of  $i \geq n$.
\end{lemma}
When $\alpha=0$, we shall write $ \mathrm{Pr}_{\infty}^{\mathrm{ord}}$.\\

We shall now construct  the improved Eisenstein family. Fix $k_0 \geq 2$ and $s_0=k_0 -2 $. It is easy to see from (\ref{bP}) and (\ref{recformula}) that when $k$ varies $p$-adically the value $\bfrak_{k}^{s_0}(I)$ varies $p$-adically analytic too. We define $\bfrak_{\kappa}^{s_0}(I) $ to be the only polynomial in  $\oo(\W(r))[T_i]$, homogeneous of degree $s_0$ such that $\bfrak_{[t]}^{s_0}(I)=\bfrak_{t+1}^{s_0}(I)$. Its coefficients are products of $\Log_p(\kappa[i])$. We let  $\chi=\chi_1\chi'\omega^{k_0-2}$ and we define
\begin{align*}
 a^*_{T_1,T_4}(\kappa)= & \left( \sum_I \bfrak_{\kappa}^{s_0}(I) {(\chi'\chi_1)}^{-1}(2T_2) \sum_{G \in \mathrm{GL}_2(\Z) \setminus \mathbf{D}(I)}(\phi\chi'\chi_1)^2(\mathrm{det}(G)){\vert \mathrm{det}(G) \vert}^{-1}{\kappa}(u^{l_{2\vert \mathrm{det}(G) \vert}}) \right. \times \\
& \left. \;\:\;\: \times L_p(\kappa,\sigma_{-\mathrm{det}(2I)}\phi\chi'\chi_1)\prod_{q \mid\mathrm{det}(2G^{-t}IG^{-1}) } B_q\left(\phi\chi'\chi_1(q)\kappa(u^{l_q})q^{-2},G^{-t}IG^{-1}\right) \right).
\end{align*}

We now construct another $p$-adic family of Eisenstein series;
\begin{align*}
  \h^*(\kappa)=\mathrm{Pr}^{\leq \alpha}{(-1)}^{k_0}\sum_{T_1 \geq 0} \sum_{T_4 \geq 0 }  a^*_{T_1,T_4}(\kappa[1-k_0]) q_1^{T_1}q^{T_2}_2.
\end{align*}

\begin{prop}
Suppose $\alpha=0$. We have a $p$-adic family $\h(\kappa,\kappa') \in {M^{(2)}_{\kappa}(N^2R^2)}^{\mathrm{ord}}$ such that 
\begin{align*}
 \h([k],\eps[t]) = \frac{1}{{(-1)}^{s} c^{s}_{t+1} A }  {(U^{\otimes^2}_{p})}^{-2i}\mathrm{Pr}^{\mathrm{ord}}  \h'_{p^i,\chi\eps\omega^{-s}}(z,w).
\end{align*}
For any $\alpha$, we have $\h^*(\kappa) \in {M^{(2)}_{\kappa}(N^2R^2)}^{\leq \alpha}$ such that
\begin{align*}
 \h^*([k]) = {A^*}^{-1}{(-1)}^{k_0}\mathrm{Pr}^{\leq \alpha}{\h'}^*_{1,\chi_1\chi'}(z,w)\vert_{u=\frac{1}{2}-k+k_0-1},
\end{align*}
for
\begin{align*}
 A^{*}= B(k-k_0+1) {(2 \pi i )}^{k_0-2} G(\chi'\chi_1).
\end{align*}
\end{prop}
\begin{proof}
 From its own definition we have ${(U^{\otimes^2}_{p})}^{2j}\h'_{1,\chi}(z,w) =\h'_{p^j,\chi}(z,w)$. We define 
\begin{align*}
 \h(\kappa,\kappa')= \varinjlim_{j} {(U^{\otimes^2}_{p})}^{-2j}\mathrm{Pr}^{\mathrm{ord}}\h_{p^j,\chi}(\kappa,\kappa').
\end{align*}
With Lemma \ref{lemmaprojinfty} and the previous remark we obtain
\begin{align*}
4^{s}{(U^{\otimes^2}_{p})}^{-2i}\mathrm{Pr}^{\mathrm{ord}} \h'_{p^i,\chi\eps\omega^{-s}}(z,w) = & 4^{s}{(U^{\otimes^2}_{p})}^{-2j}\mathrm{Pr}^{\mathrm{ord}} \h'_{p^j,\chi\eps\omega^{-s}}(z,w) \\
\equiv & A {(-1)}^s 4^{s} c^{s}_{t}  {(U^{\otimes^2}_{p})}^{-2j}\mathrm{Pr}^{\mathrm{ord}} \h_{p^j}([k],\eps[t])   \bmod p^{2j},
\end{align*}
as $U_p^{-1}$ acts on the ordinary part with norm $1$.\\
The punctual limit is then a well-defined classical, finite slope form. By the same method of proof of  \cite[Corollary 3.4.7]{UrbNholo} or \cite[Proposition 2.17]{RosOC} (in particular, recall that when the slope is bounded, the ray of overconvergence $v$ can be fixed), we see that this $q$-expansion defines a  family of finite slope forms.\\
For the second family, we remark that $\Log_p(\kappa)$ is bounded on $\W(r)$, for any $r$ and we reason as above.
\end{proof}
We want to explain briefly why the construction above works in the ordinary setting and not in the finite slope one.\\
 It is slightly complicated to explicitly  calculate the polynomial $\bfrak_{t+1}^{s}(I)$ and in particulat to show that they vary $p$-adically when varying $s$. But we know that $\mathring{\Dfrak_l^{s}}$ is an homogeneous polynomial in $\partial_i$ of degree $s$. Suppose for now $\alpha=0$, {\it i.e.} we are in the ordinary case. We have a single monomial of $\mathring{\Dfrak_l^{s}}$ which does not involve $\partial_1$ and $\partial_4$, namely $ c^{s}_{t+1} \partial_2^{s}$. Consequently, in $\bfrak_{t+1}^{s}(I)$ there is a single monomial without $T_1$ and $T_4$.\\
 When  the entries on the diagonal of $I$ are divisible by $p^i$, $4^{s}\bfrak_{t+1}^{s}(I)$ reduces to $ (-1)^s 4^{s} c^{s}_{t+1}$ modulo $p^i$.  Applying $U^{\otimes^2}_{p}$ many times ensures us that $T_1$ and $T_4$ are very divisible by $p$. Speaking $p$-adically, we approximate $\mathring{\Dfrak_l^{s}}$ by $\partial_2$ (multiplied by a constant). The more times we apply $U^{\otimes^2}_{p}$,  the better we can approximate $p$-adically $\mathring{\Dfrak_l^{s}}$ by $\partial^s_2$. At the limit, we obtain equality.\\
 
 For $\alpha >0$,  when we apply ${U^{\otimes^2}_{p}}^{-1}$ we introduce denominators of the order of $p^{\alpha}$. In order to construct a two-variable family, we should approximate  $\mathring{\Dfrak_l^{s}}$ with higher precision. For example, it would be enough to consider the monomials $\partial^J$ for $j_1+j_4 \leq \alpha$. In fact,  $\partial_1 $ and $\partial_4$ increase the slope of $U_p$; we have the relation 
 \begin{align*}
  \left(U_p^{\otimes^2}\partial_1 \partial_4\right)_{\vert z_2=0} = p^2\left(\partial_1 \partial_4 U_p^{\otimes^2} \right)_{\vert z_2=0}.
 \end{align*}
Unluckily, it seems quite hard to determine explicitly the coefficients $c_{t+1}^J$ of (\ref{polyDfrak}) or even show that they satisfies some $p$-adic congruences as done in \cite{CP,Gorsse,RosOC}. We guess that it could be easier to interpolate $p$-adically the projection to the ordinary locus of the Eisenstein series of \cite[Theorem 3.1]{BS} rather than the holomorphic projection as we are doing here (see also \cite[\S 3.4.5]{UrbNholo} for the case of nearly holomorphic forms for $\gl$).\\
 This should remind the reader of the fact that on $p$-adic forms  the Maa\ss{}-Shimura operator and Dwork $\Theta$-operator coincide \cite[\S 3.2]{UrbNholo}.

\section{$p$-adic $L$-functions}\label{padicLfunctions}
We now construct  two $p$-adic $L$-functions using the above Eisenstein measure: the two-variable one in the ordinary case and the improved one for any finite slope. Necessary for the construction is a $p$-adic Petersson product \cite[\S 6]{Pan} which we now recall. We fix a family $F=F(\kappa)$ of finite slope modular forms which we suppose primitive, i.e. all its classical specialization are primitive forms, of prime-to-$p$ conductor $N$. We consider characters $\chi$, $\chi'$ and $\chi_1$ as in Section \ref{EisenGSp4}. We keep the same decomposition for $N$ (because the local behavior at $q$ is constant along the family). We shall write $N_0$ for the conductor of $\chi_1$.
\begin{center}
 (\textbf{notCM}) We suppose that $F$ has not complex multiplication by $\chi$.
\end{center}
  Let $\Cc_F$ be the corresponding  irreducible component of the Coleman-Mazur eigencurve. It is finite flat over $\W(r)$, for a certain $r$. Let $\I$ be the coefficients ring of $F$ and $\K$ its field of fraction. For a classical form $f$, let us denote by $f^c$ the complex conjugated form. We denote by $\tau_N$ the Atkin-Lehner involution of level $N$ normalized as in \cite[h4]{H6}. When the level is clear from the context, we shall simply write $\tau$. \\
Standard linear algebra allows us to define a $\K$-linear form $l_F$ on ${M_{\kappa}(N)}^{\leq \alpha}\otimes_{\I}\K$ with the following property \cite[Proposition 6.7]{Pan};
\begin{prop}
 For all $G(\kappa)$ in ${M_{\kappa}(N)}^{\leq \alpha}\otimes_{\I}\K$ and any $\kappa_0$ classical point, we have
\begin{align*}
 l_F(G(\kappa))_{\vert \kappa=\kappa_0} = \frac{\lla F(\kappa_0)^{c}|\tau_{Np}, G(\kappa_0) \rra_{Np}}{\lla F(\kappa_0)^{c}|\tau_{Np}, F(\kappa_0) \rra_{Np}}.
\end{align*}
\end{prop}
We can find $H_F(\kappa)$ in $\I$ such that $H_F(\kappa)l_F$ is defined over $\I$. We shall refer sometimes to $H_F(\kappa)$ as the {\it denominator} of $l_F$. \\
We define consequently a $\K$-linear form  on  ${M^{(2)}_{\kappa}(N)}^{\leq \alpha}\otimes_{\I}\K$ by 
\begin{align*}
 l_{F \times F}:= l_F \otimes l_F
\end{align*}
under the decomposition in (\ref{2varslope}).\\

Before defining the $p$-adic $L$-functions, we need an operator to lower the level of the Eisenstein series constructed before. We follow \cite[\S 1 VI]{H1bis}. Fix a prime-to-$p$ integer $L$, with $N|L$. We define for classical weights $k$:
\begin{align*}
\begin{array}{ccccc}
 T_{L/N,k} : & M_k(Lp,A) & \rightarrow & M_k(Np,A)\\
 & f & \mapsto &{(L/N)}^{k/2} \sum_{[\gamma] \in \G(N)/ \G(N,L/N)}  f |_k \left(\begin{array}{cc} 
1 & 0 \\
0 & L/N
\end{array} \right) |_k \gamma
\end{array}.
\end{align*}
As $L$ is prime to $p$, it is clear that $T_{L/N,k}$ commutes with $U_p$. It extends uniquely to a linear map
$$\begin{array}{ccccc}
 T_{L/N} : & {M_{\kappa}(L)} & \rightarrow & {M_{\kappa}(N)}
\end{array}
$$
which in weight $k$ specializes to $T_{L/N,k}$.\\

 We have a map ${M_{\kappa}(N)}^{\leq \alpha} \hookrightarrow {M_{\kappa}(N^2R^2)}^{\leq \alpha}$. We define  $1_{N^2R^2/N}$ to be one left inverse. We define 
\begin{align*}
 L_p(\kappa,\kappa') & = \frac{N_0}{N^2R^2N_1}l_F \otimes l_F\left((U^{-1}_{N^2/N_1^2}\circ 1_{N^2R^2/N}) \otimes T_{N^2R^2/N}(\h(\kappa,\kappa'))\right) (\mbox{for } \alpha=0),\\
L^*_p(\kappa)        & = \frac{N_0}{N^2R^2N_1} l_F \otimes l_F\left((U^{-1}_{N^2/N_1^2}\circ 1_{N^2R^2/N})  \otimes T_{N^2R^2/N}(\h^*(\kappa)\right) (\mbox{for } \alpha < +\infty).
\end{align*}
We will see in the proof of the following theorem  that it is independent of the left inverse  $1_{N^2R^2/N}$ which we have chosen.\\

We fix some notations. For a Dirichlet character $\eta$, we denote by $\eta_0$ the associated primitive character. Let $\lambda_p(\kappa) \in \I$ be the $U_p$-eigenvalue of $F$. We say that $(\kappa, \kappa') \in \Cc_{F} \times \W$ is of type $(k; t,\eps)$ if :
\begin{itemize}
 \item $\kappa_{\vert_{\W(r)}}=[k]$ with $k\geq 2$,
 \item $\kappa'=\eps[t]$ with $1 \leq t \leq k-1$ and $\eps$ finite order character defined modulo $p^n$, $n\geq 1$.
\end{itemize}
Let as before $s=k-t-1$. We define
\begin{align*} 
E_1(\kappa,\kappa')= & \lambda_p(\kappa)^{-2n_0}(1-{(\chi\eps\omega^{-s})}_0(p)\lambda_p(\kappa)^{-2}p^{s}))
\end{align*}
where $n_0=0$ (resp. $n_0=n$) if $\chi\eps\omega^{-s}$ is (resp. is not) trivial at $p$.\\ 
If $F(\kappa)$ is primitive at $p$ we define $E_2(\kappa,\kappa')=1$, otherwise
\begin{align*}
E_2(\kappa,\kappa')= & (1-{(\chi^{-1}\eps^{-1}\omega^{s}\phi)}_0(p)p^{k-2-s}) \times \\
 & (1-{(\chi^{-1}\eps^{-1}\omega^{s}\phi^{2})}_0(p)\lambda_p(\kappa)^{-2} p^{2k-3-s}).
\end{align*}
 We  denote by $F^{\circ}(\kappa)$ the primitive form associated to $F(\kappa)$. We shall write $W'(F(\kappa))$ for the prime-to-$p$ part of the root number of $F^{\circ}(\kappa)$. If $F(\kappa)$ is not $p$-primitive we pose 
\begin{align*}
S(F(\kappa)) = 
(-1)^k   \left( 1 - \frac{\phi_0(p)p^{k-1}}{\lambda_p(\kappa)^{2}} \right)\left( 1 - \frac{\phi_0(p)p^{k-2}}{\lambda_p(\kappa)^{2}} \right),
\end{align*}
and $S(F(\kappa)) = (-1)^k  $ otherwise. 
We pose 
\begin{align*}
C_{\kappa,\kappa'} &  =i^{1-k} s ! G\left({(\chi\eps\omega^{-s})}^{-1}\right) {(\chi\eps\omega^{-s})}_0(p^{n-n_0}) {(N_1Rp^{n_0})}^{s} N^{-k/2} 2^{-s},\\
C_{\kappa} & = C_{\kappa,[k-k_0+1]}, \\
\Omega(F(\kappa),s) & = W'(F(\kappa)){(2\pi i)}^{s+1}\lla F^{\circ}(\kappa),F^{\circ}(\kappa)\rra.
\end{align*}
We have the following theorem, which will be proven at the end of the section.
\begin{theo}\label{T1}
\begin{itemize}
\item[i)]
 The function $L_p(\kappa,\kappa')$ is defined  on $\Cc_{F} \times \W$, it is meromorphic in the first variable and bounded  in the second variable. For all classical points $(\kappa, \kappa')$ of type $(k; t,\eps)$ with $k\geq 2$, $1 \leq t \leq k-1$, we have the following interpolation formula
\begin{align*}
L_p(\kappa,\kappa') =  C_{\kappa,\kappa'} E_1(\kappa,\kappa')E_2(\kappa,\kappa') \frac{\Ll(s+1,\mathrm{Sym}^2(F(\kappa)),\chi^{-1}\eps^{-1}\omega^{s})}{S(F(\kappa))\Omega(F(\kappa),s)}.
\end{align*}
\item[ii)] 
The function $L_p^*(\kappa)$ is meromorphic on $\Cc_{F}$.  For $\kappa$ of type $k$ with $ k \geq k_0 $, we have the following interpolation formula
\begin{align*}
  L^*_p(\kappa) =  C_{\kappa} E_2(\kappa,[k-k_0+1]) \frac{\Ll(k_0-1,\mathrm{Sym}^2(F(\kappa)),{\chi'}^{-1}\chi_1^{-1})}{S(F(\kappa))\Omega(F(\kappa),k_0-2)}.
\end{align*}
\end{itemize}
\end{theo}
Let us denote by $\tilde{L}_p(\kappa,\kappa')$ the two-variable $p$-adic $L$-function of \cite[Theorem 4.14]{RosOC}, which is constructed for any slope but NOT for $p=2$.
We can deduce the fundamental corollary which allows us to apply the method of Greenberg and Stevens;
\begin{coro}\label{CoroImp}
For $\alpha=0$ (resp. $\alpha > 0$ and $p\neq 2$)  we have the following factorization of locally analytic functions around $\kappa_0$ in $\Cc_F$:
\begin{align*}
L_p(\kappa,[k-k_0+1]) & = (1-\chi'\chi_1(p)\lambda_p(\kappa)^{-2}p^{k_0-2})L_p^*(\kappa)\\
(\mbox{resp. } C \frac{\kappa^{-1}(2)\tilde{L}_p(\kappa,[k-k_0+1])}{1-\phi^{-2}\chi^2 \kappa(4)2^{-2k_0}}& = (1-\chi'\chi_1(p)\lambda_p(\kappa)^{-2}p^{k_0-2})L_p^*(\kappa)),
\end{align*}
where $C$ is a constant independent of $\kappa$, explicitly determined by the comparison of $C_{\kappa}$ here and in \cite[Theorem 4.14]{RosOC}.
\end{coro} 
We can now prove the main theorem of the paper;
\begin{proof}[Proof of  Theorem \ref{MainThBon}]
Let $f$ be as in the statement of the theorem; we take $\chi'=\chi_1=\mathbf{1}$. For $\alpha=0$ resp. $\alpha > 0$ we define
\begin{align*}
L_p(\mathrm{Sym}^2(f),s) = C^{-1}_{\kappa_0,[1]}L_p(\kappa_0,[k_0-s]),\\
L_p(\mathrm{Sym}^2(f),s) =  C \frac{ \kappa(2)\tilde{L}_p(\kappa,[k-k_0+1])}{C_{\kappa_0,[1]}(1-\phi^{-2}\chi^2 \kappa(4)2^{-2k_0})}
\end{align*} 
The two variables $p$-adic $L$-function vanishes on $\kappa'=[1]$. As $f$ is Steinberg at $p$, we have $\lambda_p(\kappa_0)^{2}=p^{k_0-2}$.\\
 Consequently, the following formula is a straightforward consequence of Section \ref{GS} and Corollary \ref{CoroImp};
\begin{align*}
 \lim_{s \rightarrow 0} \frac{L_p(\mathrm{Sym}^2(f),s)}{s} =  -2 \frac{\textup{d} \log \lambda_p(\kappa)}{\textup{d}\kappa}\vert_{\kappa=\kappa_0}  \frac{\Ll(\mathrm{Sym}^2(f),k_0-1)}{S(F(\kappa))\Omega(f,k_0-2)}.
\end{align*}
From \cite{BenLinv2,MokLinv} we obtain
\begin{align*}
 \Ll^{\mathrm{al}}(\mathrm{Sym}^2(f)) = & -2 \frac{\textup{d} \log \lambda_p(\kappa)}{\textup{d}\kappa}\vert_{\kappa=\kappa_0}.
\end{align*}
Under the hypotheses of the theorem, $f$ is Steinberg at all primes of bad reduction and we see from \cite[\S 3.3]{RosOC} that 
\begin{align*}
 \Ll(\mathrm{Sym}^2(f),k_0-1) = L(\mathrm{Sym}^2(f),k_0-1)
\end{align*}
and we are done.
\end{proof}
\begin{proof}[Proof of Theorem \ref{T1}]
 We point out that most of the calculations we need in this proof and have not already been quoted can be found in \cite{H6,Pan}. \\
If $\eps$ is not trivial at $p$, we shall write $p^n$ for the conductor of $\eps$. If $\eps$ is trivial, then we let $n=1$.\\
We recall that $s=k-t-1$; we have 
\begin{align*}
L_p(\kappa,\kappa') = &  \frac{N_0 \lla  F(\kappa)^c|\tau_{Np} , U^{-1}_{N^2/N_1^2} \lla F(\kappa)^c|\tau_{Np}, T_{N^2R^2/N,k} \h([k],\eps[t])    \rra\rra}{N^2R^2 N_1 {\lla F(\kappa)^c | \tau, F(\kappa) \rra}^2 }. 
\end{align*}
We have as in \cite[(7.11)]{Pan}
\begin{align}\label{PanUpvalue}
 \lla F(k)^c|\tau_{Np}, U_p^{-2 n +1} \mathrm{Pr}^{\mathrm{ord}} U_p^{2 n-1} g   \rra = \lambda_p(\kappa)^{1-2n}p^{(2n-1)(k-1)} \lla F(\kappa)^c|\tau_{Np}|[p^{2n-1}],  g   \rra, 
\end{align}
where $f|[p^{2n-1}](z)=f(p^{2n-1}z)$. We recall the well-known formulae \cite[page 79]{H1bis}:
\begin{align*}
 \lla f|[p^{2n-1}] , T_{N^2R^2/N,k} g   \rra = & {(NR^2)}^k \lla f|[p^{2n}NR^2], g   \rra, \\
\tau_{Np}|[p^{2n-1}N^2R^2] =  &{\left(N R^2{p^{2n-1}}\right)}^{-k/2}\tau_{N^2R^2p^{2n}},\\
  \frac{\lla F(\kappa)^{c}|\tau_{Np}, F(\kappa) \rra}{\lla F(\kappa)^{\circ}, F(\kappa)^{\circ} \rra}  = &(-1)^k W'(F(\kappa)) p^{(2-k)/2} \lambda_p(\kappa) \times \\
& \left( 1 - \frac{\phi_0(p)p^{k-1}}{\lambda_p(\kappa)^{2}} \right)\left( 1 - \frac{\phi_0(p)p^{k-2}}{\lambda_p(\kappa)^{2}} \right).
\end{align*}
So we are left to calculate 
\begin{align*}
\frac{N_0{(NR^2)}^{k/2}\lambda_p(\kappa)^{1-2n}p^{(2n-1)\left(\frac{k}{2}-1\right)}}{A {(-1)}^s c_{t+1}^{s}N_1 N^2R^2} \lla   F(\kappa)^c|\tau_{Np}, U^{-1}_{N^2/N_1^2} \lla F(\kappa)^c|\tau_{N^2R^2p^{n}}, \h'_{1,\chi\eps\omega^{-t}}(z,w)\vert_{s=\frac{1}{2}-t}    \rra\rra.
\end{align*}
 We use Proposition \ref{InteExpr} and  \cite[(3.29)]{BS} (with the notation of {\it loc. cit.} $\beta_1=\frac{\lambda_p^{2}(\kappa)}{p^{k-1}}$) to obtain that  the interior Petersson product is a scalar multiple of  
\begin{align}\label{InteriorPro}
\frac{\Ll(\mathrm{Sym}^2(f),s +1)\tilde{E}(\kappa,\kappa')E_2(\kappa,\kappa')F(\kappa)\vert U_{N^2/N_1^2}}{\lla F(\kappa)^{\circ}, F(\kappa)^{\circ} \rra  W'(F(\kappa)) p^{(2-k)/2} \lambda_p(\kappa) S(F(\kappa))}.
\end{align}
Here
\begin{align*}
 \tilde{E}(\kappa,\kappa') = -p^{-1}(1-\chi^{-1}\eps^{-1}\omega^{s}(p)\lambda_p(\kappa)^2 p^{t+1-k})
\end{align*}
if $\chi\eps\omega^{-s}$ is trivial modulo $p$, and $1$ otherwise.   The factor $E_2(\kappa,\kappa')$ appears because $F(\kappa)$ could not be primitive. Clearly it is independent of $1_{N^2R^2/N}$ and we have $l_F(F(\kappa))=1$. We explicit the constant which multiplies (\ref{InteriorPro}); 
\begin{align*}
 & {(-1)}^s \frac{ N_0 {(NR^2)}^{k/2}\lambda_p(\kappa)^{1-2n}p^{(2n-1)\left(\frac{k}{2}-1\right)} N_1^{s+1}{(Rp^{n}) }^{s +3  -k}N^{2-k}}{ B(t) {(2 \pi i )}^{s}  c_{t+1}^{s} G\left({(\chi\eps\omega^{-s})}\right)N^2R^2N_1} 
\frac{{\omega^{k-t-1}(-1)}\Omega_{k,s}({s_1}) p_{s_1}(t+1)}{d_{s_1}(t+1)} \\
= &  {(-1)}^{-\frac{k}{2}}   \frac{N^{-\frac{k}{2}}N_0 N_1^{s}R^{s+1}\lambda_p(\kappa)^{1-2n} p^{(2-k)/2} {p}^{n( s +1)} (2s)! 2^{-2s}2^{2t} \pi^{3/2}}{2^{1+2t} \pi^{\frac{5}{2}}{(2 \pi i )}^{s} G\left({(\chi\eps\omega^{-s})}\right)  2^{-s}\frac{(2s)!}{s!}}.
\end{align*}
If $\eps_1\eps\omega^{-s}$ is not trivial we obtain 
\begin{align*}
= & i^{1-k} \frac{ N^{-\frac{k}{2}}N_1^s  R^{s}\lambda_p(\kappa)^{1-2n} p^{(2-k)/2} {p}^{n s} s! G\left({(\chi\eps\omega^{-s})}^{-1}\right)}{{(2 \pi i )}^{s+1} 2^{s} },
\end{align*}
otherwise 
\begin{align*}
= i^{1-k} \frac{ N^{-\frac{k}{2}}N_1^sR^{s}\lambda_p(\kappa)^{-1} p^{(2-k)/2} {p}^{s+1} s! G\left({(\chi_1\chi')}^{-1}\right)}{{(2 \pi i )}^{s+1}  2^{s} \chi_1\chi'(p)}.
\end{align*}

The calculations for $L_p^*(\kappa)$ are similar. We have to calculate 
\begin{align*}
\frac{N_0{(NR^2)}^{k/2}}{A^*N^2R^2N_1} \lla   F(\kappa)^c|\tau_{Np} , U^{-1}_{N^2/N_1^2} \lla F(\kappa)^c|\tau_{N^2R^2p},  {(-1)}^{k_0}{\h'}^*_{1,\chi_1\chi'}(z,w)\vert_{u=s_1}    \rra\rra,
\end{align*}
where $s_1 =\frac{1}{2} - k + k_0 -1$. The interior Petersson product equals (see \cite[Theorem 3.1]{BS} with $M$, $N$ of {\it loc. cit.} as follows: $M=R^2 N^2 p$, $N=N_1R$)
\begin{align*}
 \frac{R^{k_0+1-k} {(N^2p)}^{\frac{2-k}{2}} {(-1)}^{k_0}\Omega_{k,k_0-2}({s_1}) p_{s_1}(k - k_0+2)}{d_{s_1}(k - k_0+2) \lla F(\kappa)^{\circ}, F(\kappa)^{\circ} \rra  W'(F(\kappa)) p^{(2-k)/2} \lambda_p(\kappa) S(F(\kappa))} F(\kappa)\vert U_p \vert U_{N^2/N_1^2},
\end{align*}
so we have 
\begin{align*}
 L^*_p(\kappa) =  & i^{1-k}\frac{ N_1^{k_0-2}R^{k_0-2}N^{k/2} (k_0-2)!}{ {(2 \pi i )}^{k_0-1} G(\chi_1\chi') 2^{k_0-2}}  \frac{ E_2(\kappa,[k-k_0+1])\Ll(k_0-1,\mathrm{Sym}^2(F(\kappa)),{(\chi_1\chi')}^{-1})}{S(F(\kappa))W'(F(\kappa)){\lla F(\kappa)^{\circ}, F(\kappa)^{\circ} \rra}}.
\end{align*}
\end{proof}
We give here some concluding remarks.  As $F(\kappa)$ has not complex multiplication by $\chi$, we can see exactly as in \cite[Proposition 5.2]{H6} that $H_F(\kappa)L_p(\kappa,\kappa')$ is holomorphic along $\kappa'=[0]$ (which is the pole of the Kubota Leopoldt $p$-adic $L$-function).\\ We point out that the analytic $\Ll$-invariant for $CM$ forms has already been studied in literature \cite{DD,HarCM,HarLei}.\\
Note also that our choice of periods is not optimal \cite[\S 6]{RosOC}.

\bibliographystyle{alpha}
\bibliography{Bibliografy}
\end{document}